\documentclass[%11pt,
reqno]{amsart}
\usepackage{amscd,amsmath,amssymb,amsthm,amsfonts}
\usepackage{lipsum}
\usepackage{xpatch}
\xpatchcmd{\paragraph}{\normalfont}{{\normalfont\bfseries}}{}{}
\usepackage{cancel}
\usepackage{color,xcolor}
\usepackage{graphics}
\usepackage{inputenc}
\usepackage{fontenc}

\usepackage[english]{babel}
\newcommand{\myspace}{\qquad\qquad\qquad}

\newcommand{\cA}{{\mathcal A}}
\newcommand{\cB}{{\mathcal B}}
\newcommand{\cD}{{\mathcal D}}

\newcommand{\cH}{{\mathcal H}}

\newcommand{\cL}{{\mathcal L}}

\newcommand{\cO}{{\mathcal O}}
\newcommand{\cR}{{\mathcal R}}

\newtheorem{theorem}{Theorem}[section]

\newtheorem{proposition}[theorem]{Proposition}
\newtheorem{remark}[theorem]{Remark}
\newtheorem{remarks}[theorem]{Remarks}

\newtheorem{definition}[theorem]{Definition}

\newtheorem{problem}[theorem]{Problem}

\numberwithin{equation}{section}

\date{}

\begin{document}

\title{Improved boundary regularity for a Stokes-Lam\'e system}

%\title[]{Enhanced boundary regularity for a Stokes-Lam\'e system}

%\title[]{Improved boundary regularity in a fluid-elasticity interaction}

%\title[]{Improvement over Sobolev trace regularity for a Stokes-Lam\'e system}

\author{Francesca Bucci}
\address{Francesca Bucci, Universit\`a degli Studi di Firenze,
Dipartimento di Matematica e Informatica,
Via S.~Marta 3, 50139 Firenze, ITALY
}
\email{francesca.bucci@unifi.it}

\begin{abstract}
This paper recalls a partial differential equations system, which is the linearization of a recognized
fluid-elasticity interaction three-dimensional model. A collection of regularity results for the traces of
the fluid variable on the interface between the body and the fluid is established, in the case a suitable boundary dissipation is present. These regularity estimates -- in time and space, of local and global nature -- are geared toward ensuring the well-posedness of the algebraic Riccati equations which arise from
the associated optimal boundary control problems on an infinite time horizon. The theory of operator semigroups and interpolation provide the main tools.
\end{abstract}

\maketitle

% Section: INTRO

\section{Introduction}
We consider a variant of a recognized partial differential equations (PDE) system
that arises in the modeling of the interactions of an elastic body fully immersed
in a fluid.
A mathematical description of the linear PDE problem under consideration, that is
\eqref{e:dampedfsi-free} below with $a_2>0$, is given in Section~\ref{ss:intro_main}.
The dynamics of the fluid and the solid are described by the equations of Stokes flow in the variables $u$ (fluid velocity field) and $p$ (pressure), and by a system of linear elasticity in the variable $w$, respectively. 
In both the original fluid-structure interaction (FSI) problem -- namely, \eqref{e:dampedfsi-free} with 
$a_2=0$ %(as well as $a_1=0$) 
-- and the present one, the interface $\Gamma_s$ between the fluid and the solid is assumed to be fixed.
This characteristic finds a physical justification % originates in the consideration
in the fact that the motion of the solid is considered as entirely due to infinitesimal displacements (fast, though).
Instead, the two PDE systems differ inasmuch as the absence/presence of
a boundary dissipation term ($a_2=0$ {\em vs} $a_2>0$) changes -- from strong to uniform -- the stability of the corresponding dynamics,
as proved by Avalos and Triggiani \cite{avalos-trig-dampedfsi_2009}.
Another favourable consequence of the change in the transmission boundary condition is the improved regularity of the normal component of the elastic stress tensor on the interface.

Our aim in this paper is to pinpoint the resulting regularity of the boundary traces of the fluid variable.
The obtained results, collectively stated as Theorem~\ref{t:main} in Section~\ref{ss:intro_main} below,
are central to solvability of the associated optimal control problems with quadratic functionals; see Section~\ref{s:application}, in particular Remark~\ref{r:pde-theoretic} and the subsequent verification.
To be more specific, Theorem~\ref{t:main} attests the applicability of the Linear-Quadratic (LQ) theory on an infinite time horizon devised by Acquistapace, Lasiecka and this author in \cite{abl_2013}.
This theory, tailored to coupled PDE systems which comprise both hyperbolic and parabolic components, with the latter subject to boundary/interface control,
finds in \eqref{e:dampedfsi-free} a significant FSI illustration, besides and beyond relevant physical interactions such as mechanical-thermal (and acoustic-structure) ones, which provided its original motivation.

While the former (undamped) FSI occurs already in the 1969 monograph \cite{lions_69} by J.-L.~Lions, it was brought to wider attention much later by Du~{\em et al.} \cite{du-etal_2003}.
The proof of its well-posedness in a natural functional setting has been 
given by Barbu~{\em et al.}~\cite{barbu-etal_2007}.
Intrinsic features of the coupled PDE problem are 
(i) the Neumann-type boundary condition (involving the pressure as well) which hinders the
Leray projection to divergence-free spaces, and also 
(ii) the transmission condition on the interface.
The latter raises a major technical issue, that is the apparent discrepancy
between the trace regularity of the variable that describes the fluid flow and the displacement of the elastic solid.
First in the mathematical literature, the study \cite{barbu-etal_2007} -- which pertains to the actual {\em nonlinear} PDE problem, comprising the Navier-Stokes equations -- shows the existence of finite energy weak solutions
whose definition incorporates the exceptional boundary regularity of the elastic variable  
on the interface, thus solving an open problem until then.

With the focus on optimal control problems with quadratic functionals
associated with the linearized Stokes-Lam\'e system \eqref{e:dampedfsi-free}
(still with $a_2=0$), a set of trace regularity results for the fluid variables $u$ and $u_t$ have been established jointly with Lasiecka \cite{bucci-las-fsi_2010,bucci-las-fsi_2011}.
The article \cite{bucci-las-fsi_2011} provides a slightly more general treatment.
In addition, it gives the explicit proof of a preliminary result (Lemma~2.4 therein)
that pinpoints the regularity of the normal stress $\sigma(w)\cdot \nu$ on the interface
(more specifically of a component of it), when $(u,w,w_t)$ solves the {\em coupled} PDE problem 
\eqref{e:dampedfsi-free} with initial data in a scale of spaces $Y_\alpha$, 
$\alpha \in [0,1]$, and $u|_{\Gamma_s}\in H^\alpha((0,T)\times \Gamma_s)$.
The proof of the aforesaid Lemma~2.4 combines the methods of microlocal analysis utilized by Barbu~{\em et al.} in \cite[Theorem~3.3]{barbu-etal_2007} with interpolation, as done previously by Lasiecka and Tuffaha
in \cite{las-tuff-fsi_2009}.
In addition, it utilizes a well-known trace result for the solutions to the Cauchy-Dirichlet problem for wave equations 
with initial data in $H^1\times L^2$ ({\em cf.}~\cite{las-lions-trig_1986}).
(An improved `hidden' regularity result for the {\em uncoupled} Lam\'e system has been derived in 
\cite{raymond-fsi_2014} for the case of moving interface, flat though.)
\\
The boundary regularity estimates collected in \cite[Theorem~2.10]{bucci-las-fsi_2010} -- interesting by themselves -- are central to the invocation of the theory of the LQ-problem on a finite time horizon initiated by Acquistapace, Lasiecka and this author in \cite{abl_2005}.
Since the dynamics of the original FSI is not uniformly stable, the optimal 
boundary control problem on an infinite time interval had remained outside the investigation.

In this paper we devote our attention to the FSI \eqref{e:dampedfsi-free} with $a_2>0$.
Given the change in the transmission condition which renders the overall dynamics uniformly stable
({\em cf.}~\cite{avalos-trig-dampedfsi_2009}),
the theory of the infinite time horizon LQ-problem developed by Acquistapace~{\em et al.} in \cite{abl_2013}
is potentially suited to be used.
The requirement that the uncontrolled PDE system is (uniformly) exponentially stable -- a prerequisite therein -- is fulfilled.
We work within a semigroup framework which is largely consistent with the one in \cite{bucci-las-fsi_2010,bucci-las-fsi_2011}.
The coupled PDE system, subject to a control force acting on the interface, is reformulated as an abstract system $y'=\cA y+\cB g$ in a suitable Banach space $Y$, with the free dynamics generator $\cA$ and control operator $\cB$ defined in Proposition~\ref{p:fsi-well-posed}.
The regularity estimates (in time and space, of local and global nature, and of independent interest) attained
in Theorem~\ref{t:main}, once interpreted in functional-analytic terms -- i.e. as regularity properties of the adjoint of the unbounded
operator $e^{t\cA}\cB$ --, allow for the application of the LQ-theory of \cite{abl_2013}.
The implications of the (PDE) trace regularity results in well-posedness of the associated (operator)
Riccati equations will be highlighted in Section~\ref{s:application}.
It is important to emphasize that the obtained regularity results do not hinge upon smoothing observations.
This means that we allow fairly general functionals, including the integral of the full quadratic energy of the physical system; see \eqref{e:energy-cost} in Section~\ref{s:proofs}.
(For the undamped case, this advance on the earlier study \cite{las-tuff-fsi_2009} was first achieved in 
\cite{bucci-las-fsi_2010}.)

On the mathematical side, we note that not only the obtained regularity estimates 
-- that constitute the core of this work -- are improved over Sobolev trace regularity. 
Also, the one pertaining to the fluid acceleration (i.e.~\eqref{e:u_t-better-reg} of Theorem~\ref{t:main})
is enhanced, when compared to the respective one sought and established in 
\cite[Theorem~2.10]{bucci-las-fsi_2010}.
In addition, the proofs can be made somewhat simpler.
We recall that in the case of the original FSI microlocal analysis arguments were necessitated to disclose the sharp regularity results pertaining to the elastic (hyperbolic) component; these results were crucially utilized to single out the trace results for the fluid (parabolic) component.
In the present case, as a result of the presence of the boundary dissipation term,
the theories of operator semigroups and of interpolation spaces alone will suffice.

% LITERATURE

\subsection{Some bibliographical comments} 
As a complement to the first part of the Introduction, we furnish some (non-exhaustive) bibliographical notes.
The topic broadly referred to as FSI is a prominent section of the research and application area known as multiphysics.
The questions that FSI raise, and the challenges to be faced, span both the modeling and the mathematical analysis.
Because of the complexity of the PDE systems describing FSI, discretization techniques assume a fundamental role.
It goes far beyond the scopes of the present work to provide a comprehensive 
overview of the many contributions to the research advances within this subject;
a minimum information follows.
A majority of the modeling and numerical works focus on aeroelastic and hemodynamic problems,
consistently with the clear engineering and medical (cardiovascular) applications, respectively;
{\em cf.}~Quarteroni~{\em et al.} \cite{quarteroni-etal_2000}, Wick and Richter \cite{richter-wick_2013,richter-book_2017}, along with their references.
Contributions to the understanding and mathematical analysis of nonlinear FSI arising in aeroelasticity, with major focus on the well-posedness of the PDE models (besides the modeling itself)
and the long-time behaviour of the flow-plate dynamics (in particular, attractors), were given by Chueshov\footnote{Igor Chueshov, passed away in 2016, will be sadly missed and forever remembered.}.
We mention explicitly his joint work with Ryzkhova \cite{chu-ryzhk_2013}, and lastly with 
Lasiecka and Webster; the reader is referred to the review article \cite{chueshov-etal-review_2016} and its references.
% (along with praise to Balakrishnan for his pioneering and insightful work on flutter.
%(Great credit is due to A.V.~Balakrishnan, ``whose pioneering and insightful work on flutter brought %together engineers and mathematicians alike''.) 
%
See also the recent (and very nice) chapter devoted to the mathematical theory of evolutionary fluid-flow structure interactions %by Lasiecka and Webster 
\cite{las-webst-oberwolfach_2018}.

Among the former studies of the motion of rigid bodies immersed in a fluid, we recall the ones
by San~Mart\'in~{\em et al.} \cite{sanmartin-etal_2002} in 2D and by Feireisl \cite{feireisl_2003}
in 3D.
Most FSI are free boundary problems.
The local existence and uniqueness (and regularity) for a PDE system coupling the incompressible Navier-Stokes equations with the linear Kirchhoff material model for the description of an elastic solid immersed in a fluid, with {\em moving} interface, was first established by Coutand and Shkoller \cite{coutand_shkoller_2005,coutand_shkoller_2006}.
Subsequent progress in the study of the well-posedness (within this type of scenario) has been carried on by Kukavica and Tuffaha \cite{kuk-tuff_2012}, Raymond and Vanninathan \cite{raymond-fsi_2014} and more recently by Ignatova~{\em et al.} \cite{ignatova_etal_2017}; the reader is referred to \cite{ignatova_etal_2017}
and its references.
Advances in both the modeling and well-posedness of FSI
are pursued by Canic and Muha (jointly as well as independently, and with different coauthors); see, e.g., 
\cite{muha-canic_2016}.
These include fluid-rigid body interactions that are described by the coupling of
the Navier-Stokes equations with a system of ODE.  
A variety of scenarios for FSI are explored as well by Ne\v{c}asov\'a and coauthors, 
with novel choices and solutions devised; see e.g. \cite{chemetov-etal_2019}.

Moving on to optimal control problems, we recall the work \cite{las-tuff-fsi_2009} by Lasiecka and Tuffaha, where it is first showed that for the control system $y'=\cA y+\cB g$ representing the FSI \eqref{e:dampedfsi-free} (with $a_2=a_1=0$), subject to a boundary force $g$, 
a {\em singular estimate} for the norm of $\cO e^{t\cA}\cB$ is valid, $\cO$ being a suitable {\em observation} operator.
This property brings about the sought {\em feedback} form of the optimal control, provided the minimization pertains to quadratic cost functionals in a restricted class.
While the analysis in \cite{las-tuff-fsi_2009} includes the Bolza problem (i.e. allows a penalization of the state at a final time $T>0$),
it excludes the integral of the {\em full} quadratic energy of the system.

The recent paper by Bociu~{\em et al.} \cite{bociu-etal-fsicontrol_2020} deals with an optimal control
problem for a fluid-elasticity interaction with moving boundary,
in the presence of a distributed control action (in feedback form).
The existence of an optimal control is proved when the minimization specifically pertains to the vorticity inside the fluid.     
The reader is referred to \cite{bociu-etal-fsicontrol_2020} for an insight into the relevant technical issues that arise in the analysis and optimization of FSI with moving interface, as well as for a more complete
literature review (including FSI comprising {\em compressible} fluids, left out of the discussion here
for the sake of space).
To the author's knowledge, a study of the optimal control problem (with general -- quadratic or possibly
non-quadratic -- cost functionals) for the PDE system studied in \cite{barbu-etal_2007} which couples the incompressible Navier-Stokes equations with the Lam\'e system (when subject to boundary actions), thereby deriving necessary conditions for an optimal {\em open-loop} control to exist, seems still lacking.

% INTRO (cont'd): the FSI, MAIN result

\subsection{The PDE problem. Main result} \label{ss:intro_main}
Let us introduce the notation and then the coupled PDE system. 
The open, bounded and smooth domain representing the fluid-solid region is denoted
by $\Omega\subset \mathbb{R}^d$, $d=2,3$;
with $\Omega_f$ and $\Omega_s$ the (open and smooth) domains occupied by the 
fluid and the solid, respectively, then $\Omega$ is the interior of 
$\overline{\Omega}_f\cup\overline{\Omega}_s$.
The interactions occurr at an interface, say, $\Gamma_s=\partial\Omega_s$, which
is assumed stationary.
Finally, $\Gamma_f$ is the outer boundary of $\Omega_f$, namely 
$\Gamma_f=\partial\Omega_f\setminus \partial\Omega_s$.
Next, the velocity field of the fluid is represented by a vector-valued function $u$, which satisfies the
equations of Stokes flow in $\Omega_f$; the scalar function $p$ represents, as usual, the pressure.
The displacements of the solid region $\Omega_s$ are described by the variable $w$,
which satisfies the Lam\'e system of dynamic elasticity.
Thus, the PDE system in the unknown $(u,w)$ is as follows,
\begin{equation} \label{e:dampedfsi-free}
\begin{cases} 
u_t-{\rm div}\,\epsilon(u)+\nabla p= 0 & \textrm{in }\; Q_f:= (0,T)\times \Omega_f
\\
{\rm div}\, u=0 & \textrm{in }\; Q_f
\\ 
w_{tt} - {\rm div}\,\sigma(w)+a_1 w=0 &\textrm{in }\; Q_s:= (0,T)\times\Omega_s
\\
u=0 & \textrm{on }\; \Sigma_f:= (0,T)\times\Gamma_f
\\
\epsilon(u)\cdot\nu=\sigma(w)\cdot\nu + p\nu & \textrm{on }\; \Sigma_s:= (0,T)\times\Gamma_s
\\
w_t-a_2 \sigma(w)\cdot\nu=u & \textrm{on }\; \Sigma_s
\\
u(0,\cdot)=u_0 & \textrm{in }\; \Omega_f
\\
w(0,\cdot)=w_0\,, \quad w_t(0,\cdot)=w_1 & \textrm{in }\; \Omega_s\,.
\end{cases}
\end{equation}
where $\sigma$ and $\epsilon$ denote the elastic stress tensor and the strain tensor, respectively, that are
\begin{equation}
\epsilon_{ij}(w)= \frac12\Big(\frac{\partial w_i}{\partial x_j}
+\frac{\partial w_j}{\partial x_i}\Big)\,,
\qquad
\sigma_{ij}(w) = \lambda \sum_{k=1}^3 \epsilon_{kk}(w)\delta_{ij}
+ 2\mu \epsilon_{ij}(w)\,,
\end{equation}
where $\lambda, \mu$ are the Lam\'e constants and $\delta_{ij}$ is the Kronecker
symbol.
We note that $\nu=\nu(x)$ is the outward unit normal for the fluid region $\Omega_f$; accordingly, it is
pointing towards the interior of the solid region $\Omega_s$.  
%inward-pointing for the solid region $\Omega_s$.

\smallskip
The functional setup for the FSI \eqref{e:dampedfsi-free} is consistent -- in its
basic elements -- with the one of \cite{barbu-etal_2007} (and \cite{bucci-las-fsi_2010,bucci-las-fsi_2011}).
For the fluid component of the system one first introduces the space 
\begin{equation*}
\cH := \big\{ u\in [L^2(\Omega_f)]^d\colon {\rm div}\, u=0\,, \; 
u\cdot \nu\big|_{\Gamma_f}=0\big\}\,.
\end{equation*} 
With $[H^1(\Omega_s)]^d\times [L^2(\Omega_s)]^d$ as the natural energy space for
the Lam\'e system, the state (energy) space is then 
\begin{equation*}
Y := \cH\times [H^1(\Omega_s)]^d\times [L^2(\Omega_s)]^d\,.
\end{equation*}
Next, we denote by $V$ the space defined as follows:
\begin{equation*}
V := \big\{ v\in [H^1(\Omega_f)]^d \colon {\rm div}\,u=0\,, \;
u|_{\Gamma_f}=0\big\}\,.
\end{equation*}
We follow the usual notation
\begin{equation*}
(u,v)_f := \int_{\Omega_f}uv \,{\rm d}\Omega_f\,,\quad 
(u,v)_s := \int_{\Omega_s}uv \,{\rm d}\Omega_s\,,\quad
\langle u,v\rangle := \int_{\Gamma_s} uv \,{\rm d}\Gamma_s
\end{equation*}
for the various inner products.
The space $V$ is topologized with respect to the inner product given by 
\begin{equation*}
(u,v)_{1,f} := \int_{\Omega_f}\epsilon(u)\epsilon(v) d\Omega_f\,;
\end{equation*}
the corresponding (induced) norm $|\cdot|_{1,f}$ is equivalent to the 
usual $H^1(\Omega_f)$ norm, in view of Korn inequality and the Poincar\'e inequality.

\begin{remark}
{\rm 
In order to simplify the notation, we will write all the Sobolev spaces $[H^s(E)]^d$
(related to either function $u$, $w$) as $H^s(E)$.
The norm $\|\cdot\|_{H^r(D)}$ in the Sobolev space $H^r(D)$ will be shortly denoted by $|\cdot|_{r,D}$ throughout the paper; $|\cdot|_{0,D}$ will be semplified to $|\cdot|_D$.
}
\end{remark}

The energy of the system at time $t$ is (possibly neglecting a factor $\frac{1}{2}$) 
\begin{equation*}
|u(t)|_{\Omega_f}^2+(\sigma(t)),\epsilon(w(t))_s  + |w(t)|_{\Omega_s}^2 +|w_t(t)|_{\Omega_s}^2\,,
\end{equation*}
after setting $u(t)=u(t,\cdot)$, $w(t)=w(t,\cdot)$, etc. 

\bigskip
We will prove a series of trace regularity results for the FSI \eqref{e:dampedfsi-free}
with $a_2>0$, whose respective statements are detailed in the following Theorem.

% MAIN RESULT

\begin{theorem}\label{t:main}
Given the (uncontrolled) PDE system \eqref{e:dampedfsi-free} with $a_1,a_2>0$, 
let $y(t)=(u(t),w(t),w_t(t))$ be the semigroup solution corresponding to an 
initial state $y_0=(u_0,w_0,w_1)$, in accordance with part (b) of 
Proposition~\ref{p:fsi-well-posed}. 
The following assertions pertain to the regularity of the boundary traces 
of the fluid velocity field $u$ and acceleration $u_t$ on the interface $\Gamma_s$.

\smallskip
\noindent
{\bf \em i)}
Assume $y_0\in Y$.
Then $u$ can be represented as 
\begin{equation*}
u(t) = u_1(t)+u_2(t)\,,
\end{equation*}
where the first component satisfies the (pointwise in time) estimate
\begin{equation} \label{e:se-damped}
|u_1(t)|_{\Gamma_s} \le 
\frac{C\,e^{-\eta t}}{t^{1/4+\delta}}\|y_0\|_Y \qquad\quad  \forall y_0\in Y\,, \; 
\forall t>0
\end{equation}
for some positive constants $C, \eta$ and any positive $\delta<\frac{3}{4}$, 
while for any $T>0$
\begin{equation} \label{e:u_2-reg}
u_2\big|_{\Gamma_s}\in L^p(0,T;L^2(\Gamma_s)) \qquad \forall p\,, \; 1\le p<+\infty\,.
\end{equation}
In particular, one has
\begin{equation*}
u\big|_{\Gamma_s}\in L^{4-\sigma}(0,T;L^2(\Gamma_s)) + L^p(0,T;L^2(\Gamma_s)) 
\qquad \forall p\,, \; 1\le p<+\infty\,,
\end{equation*}
for arbitrarily small $\sigma$.

\smallskip
\noindent
{\bf \em ii)}
Assume $y_0\in \cD({\cA}^{\epsilon})$ for some $\epsilon\in (0,1)$.
Then, the regularity result in \eqref{e:u_2-reg} improves as follows:
\begin{equation} \label{e:u_2-reg-improved}
u_2\big|_{\Gamma_s}\in 
H^{1+\epsilon/2,1/2+\epsilon/4}(\Sigma_s)\,,
\end{equation}
which gives in particular
\begin{equation} \label{e:u_2-suffices}
u_2\big|_{\Gamma_s}\in H^{1/2+\epsilon/4}(0,T;L^2(\Gamma_s))
\subset C([0,T],L^2(\Gamma_s))\,.
\end{equation}

\smallskip
\noindent
{\bf \em iii)}
Let now $y_0\in \cD({\cA}^{1-\theta})$, with $\theta\in (0,\frac12)$
(and arbitrary $T>0$, just like in i)).
Then, the following trace regularity result pertains to the fluid
acceleration:
\begin{equation} \label{e:u_t-better-reg}
u_t\big|_{\Gamma_s}\in L^q(0,T;H^{1/2-\theta-\delta}(\Gamma_s)) 
\quad \text{for any $q< \frac{2}{2-\delta}$},
\end{equation}
continuously with respect to $y_0$; a fortiori,
\begin{equation} \label{e:u_t-reg}
u_t\big|_{\Gamma_s} \in L^q(0,T;L^2(\Gamma_s))
\quad \text{for any $q< \frac{4}{3+2\theta}$.}
\end{equation}

\end{theorem}

% SUBSECTION: OUTLINE of the paper

\smallskip
\subsection{Outline of the paper}
The structure of the paper is outlined readily.
The above Subsection~\ref{ss:intro_main} introduced the reader to the FSI and to the obtained trace regularity results, collected in Theorem~\ref{t:main}.

Section~\ref{s:proofs} contains the proof of Theorem~\ref{t:main}, that is our main result.
It is preceded by the statement of two preliminary results.
Proposition~\ref{p:fsi-well-posed} illustrates the abstract reformulation of the PDE problem, along with the definition of the operators which describe the (free and) controlled dynamics.
Proposition~\ref{p:damped-traces} clarifies the regularity of the normal component of the elastic stress tensor on the interface, which constitutes a basic element for the the proof of Theorem~\ref{t:main}.

In the last Section~\ref{s:application} we move on to the interpretation of several PDE estimates 
established in Theorem~\ref{t:main}, as suitable control-theoretic properties which are the key to 
well-posedness of the Riccati equations that arise in the study of the optimal control problems
associated with the FSI. % studied here.
We first provide an overview of the broader context, and at the end briefly discuss the technical details
for the present FSI.

% Section 2

\section{Trace regularity results} \label{s:proofs}
\subsection{Preliminaries: function spaces, variational and semigroup formulation}
Insert into the PDE problem \eqref{e:dampedfsi-free} a control `force' $g=g(t,x)$ acting upon the
interface $\Gamma_s$.
Since throughout the paper the boundary dissipation term will occur in the FSI, i.e. $a_2>0$,
for the sake of simplicity and without loss of generality we set $a_2=a_1=1$.
This yields the following initial-boundary value problem (IBVP) with non-homogeneous bounday datum $g$: 

\begin{equation} \label{e:dampedfsi-controlled}
\begin{cases} 
u_t-{\rm div}\,\epsilon(u)+\nabla p= 0 & \textrm{in $Q_f$}
\\
{\rm div}\, u=0 & \textrm{in $Q_f$}
\\ 
w_{tt} - {\rm div}\,\sigma(w)+w=0 &\textrm{in $Q_s$}
\\
u=0 & \textrm{on $\Sigma_f$}
\\
\epsilon(u)\cdot\nu=\sigma(w)\cdot\nu + p\nu +g & \textrm{on $\Sigma_s$}
\\
w_t-\sigma(w)\cdot\nu=u & \textrm{on $\Sigma_s$}
\\
u(0,\cdot)=u_0 & \textrm{in $\Omega_f$}
\\
w(0,\cdot)=w_0\,, \quad w_t(0,\cdot)=w_1 & \textrm{in $\Omega_s$.}
\end{cases}
\end{equation}
Then, we associate to \eqref{e:dampedfsi-controlled} a {\em general} quadratic functional to be minimized overall $g\in L^2(0,T;L^2(\Gamma_s))$, such as the one involving the integral (on a time interval $(0,T)$, possibly with $T=+\infty$) of the full energy of the system, that is
\begin{equation}\label{e:energy-cost}
\int_0^T \Big(|u(t,\cdot)|_{\Omega_f}^2+(\epsilon(w(t,\cdot)),\sigma(w(t,\cdot)))_s
+ |w(t,\cdot)|_{\Omega_s}^2+|w_t(t,\cdot)|_{\Omega_s}^2+ |g(t,\cdot)|_{\Gamma_s}^2\Big)\, dt\,.   
\end{equation}
We recall that while the regularity analysis to be carried out -- eventually culminating in Theorem~\ref{t:main} -- actually pertains the {\em uncontrolled} (or {\em free}) PDE problem, i.e.~to \eqref{e:dampedfsi-controlled} with $g\equiv 0$, the obtained trace regularity results render feasible the sought-after {\em closed-loop} form of the optimal control $\hat{g}$ minimizing the cost functional \eqref{e:energy-cost}.

To translate the boundary control system \eqref{e:dampedfsi-controlled} into an abstract equation
we follow the same avenue taken in \cite{bucci-las-fsi_2010,bucci-las-fsi_2011}, whose starting point is the variational formulation of the PDE problem.
The resulting abstract framework is suited to the computations to be carried out in order to attain the trace regularity estimates.
The definition of {\em weak} solutions to the uncontrolled version of PDE system
\eqref{e:dampedfsi-controlled} is akin to the one introduced in \cite{barbu-etal_2007}.

% WEAK SOLUTIONS and WELL-POSEDNESS

\begin{definition}
[Weak solution] % , \cite{barbu-etal_2007}] 
\label{def:weak-sol}
Let $(u_0,w_0,w_1)\in Y$ and $T>0$. 
A triple 
\begin{equation*}
(u,w,w_t)\in C([0,T],\cH\times H^1(\Omega_s)\times L^2(\Omega_s))
\end{equation*}
is said a weak solution to the PDE system \eqref{e:dampedfsi-controlled} if
\begin{itemize}

\item
$(u(0,\cdot),w(0,\cdot),w_t(0,\cdot))=(u_0,w_0,w_1)$,

\smallskip

\item
$u\in L^2(0,T;V)$,

\smallskip

\item

$\sigma(w)\cdot \nu \in L^2(0,T;H^{-1/2}(\Gamma_s))$, 
$\frac{d}{dt}w|_{\Gamma_s}-\sigma(w)\cdot \nu\big|_{\Gamma_s}
=u|_{\Gamma_s}\in L^2(0,T;H^{1/2}(\Gamma_s))$; %\,;

\smallskip

\item
the following variational system holds a.e. in $(0,T)$:
\begin{equation} \label{e:variational-system}
\begin{cases}
\frac{d}{dt}(u,\phi)_{f} + (\epsilon(u),\epsilon(\phi))_{f}
- \langle\sigma(w)\cdot \nu + g,\phi\rangle=0
\\[1mm]
\frac{d}{dt}(w_t,\psi)_{s} + (\sigma(w),\epsilon(\psi))_{s}
-\langle\sigma(w)\cdot \nu,\psi\rangle+ (w,\psi)_s=0\,,
\end{cases}
\end{equation}
for all test functions $\phi\in V$ and $\psi\in H^1(\Omega_s)$.
\end{itemize}
\end{definition}
For the sake of completeness and the reader's convenience we observe that the given definition is justified as follows.
Multiply the equation for the fluid variable by a given test function $\phi\in V$,
thereby obtaining
\begin{equation}\label{e:variational-u} 
(u_t,\phi)_f-(\text{div}\,\epsilon(u),\phi)_f +(\nabla p,\phi)_f=0\,.
\end{equation}
Integration by parts (via the Green formulas) in the summand in the middle
yields (omitting the symbols ${\rm d}\Omega_f$, ${\rm d}\Omega_s$, ${\rm d}\Gamma_s$), ${\rm d}\Gamma_f$)

\begin{equation*}
\begin{split}
&(\text{div}\,\epsilon(u),\phi)_f -(\nabla p,\phi)_f
= \int_{\Omega_f}\phi \,\text{div}\,\epsilon(u)-\int_{\Omega_f} \nabla p\,\phi
\\[1mm]
& \quad 
=\int_{\partial\Omega_f}\phi \,\epsilon(u)\cdot \nu 
- \int_{\Omega_f}\epsilon(u)\epsilon(\phi)
- \int_{\Omega_f}\big(\text{div}(p\phi)-p\text{div}\,\phi\big)
\\[1mm]
& \quad 
=\int_{\Gamma_s}\phi \,\epsilon(u)\cdot \nu
-\int_{\Omega_f}\epsilon(u)\epsilon(\phi) 
-\int_{\partial\Omega_f}\phi p\nu
-\bcancel{\int_{\Omega_f}p\, \text{div}\,\phi}
\\[1mm]
& \quad 
=\int_{\Gamma_s}(\sigma(w)\cdot \nu +p\nu+g)\phi
-\int_{\Omega_f}\epsilon(u)\epsilon(\phi)
- \int_{\Gamma_s}\phi p\nu-\bcancel{\int_{\Gamma_f}\phi p\nu}
\\[1mm]
& \quad 
=\int_{\Gamma_s}(\sigma(w)\cdot \nu +g)\phi
-\int_{\Omega_f}\epsilon(u)\epsilon(\phi)
=-\big[(\epsilon(u),\epsilon(\phi))_f
-\langle\sigma(w)\cdot \nu+g,\phi\rangle\big]\,.
\end{split}
\end{equation*}
In the above computations the distinct boundary conditions (BC) 

\begin{equation*} 
\begin{cases}
u=0  & \text{on $\Gamma_f$ (non-slip BC)}
\\[1mm]
\epsilon(u)\cdot\nu=\sigma(w)\cdot\nu+p\nu+g
& \text{on $\Gamma_s$ (transmission BC),}
\end{cases}
\end{equation*}
have been taken into account, whilst $\text{div}\,\phi=0$ and $\phi|_{\Gamma_f}=0$, because $\phi\in V$.
Thus, the reformulation of \eqref{e:variational-u} which embodies the BC
reads as
\begin{equation}\label{e:equivalent} 
\frac{d}{dt}(u,\phi)_f+(\epsilon(u),\epsilon(\phi))_f
-\langle\sigma(w)\cdot \nu+g,\phi\rangle=0\,,
\end{equation}
that is nothing but the first one of \eqref{e:variational-system}.
The computations pertaining to the equation for the solid variable $w$
are even more straightforward.

\smallskip
We proceed as in \cite{bucci-las-fsi_2010,bucci-las-fsi_2011}, which in turn
followed \cite{las-tuff-fsi_2009}.
We introduce the fluid dynamics operator $A\colon V\to V'$ defined by 
\begin{equation}\label{e:fluid-operator}
(Au,\phi) = -(\epsilon(u),\epsilon(\phi))
\qquad \forall \,\phi\in V\,,
\end{equation}
as well as the (Neumann) map $N\colon L^2(\Gamma_s)\to \cH$ defined as follows,
\begin{equation}\label{e:neumann-map}
Ng=h \Longleftrightarrow
(\epsilon(h),\epsilon(\phi))= \langle g,\phi\rangle \qquad \forall \phi\in V\,.
\end{equation}
These operators are the key elements for the semigroup formulation of the IBVP \eqref{e:dampedfsi-controlled}.
We recall from \cite[Section~4]{las-tuff-fsi_2009} that $A$ defined by \eqref{e:fluid-operator} may be considered as acting on $\cH$, with domain 
$\cD(A):=\{\,u\in V\colon \,; |(\epsilon(u),\epsilon(\phi))|\le C|\phi|_{\cH}\,\}$.
Thus, $A\colon\cD(A)\subset \cH\to \cH$ is a self-adjoint, negative operator and 
therefore is the infinitesimal generator an {\em analytic} semigroup $e^{At}$,
$t\ge 0$, on $\cH$.
Therefore, the fractional powers of $-A$ are well defined;
to simplify the notation, we shall denote them by $A^\alpha$ (rather than by $(-A)^\alpha$)
throughout.

On the basis of the respective definitions \eqref{e:fluid-operator} and \eqref{e:neumann-map} of the operators $A$ and $N$, we see that  
\begin{equation*}
\langle\sigma(w)\cdot \nu+g,\phi\rangle=
(\epsilon(N(\sigma(w)\cdot \nu+g)),\epsilon(\phi))_f
=-(AN (\sigma(w)\cdot \nu+g),\phi)\,,
\end{equation*}
which inserted in \eqref{e:equivalent} yields

\begin{equation*}
(u_t,\phi)= (Au,\phi)_f-(AN(\sigma(w)\cdot \nu+g),\phi)_f\,.
\end{equation*}
The principal facts concerning the abstract setup for problem \eqref{e:dampedfsi-controlled}, including the semigroup well-posedness for the uncontrolled problem, are collected in the following Proposition. 

\begin{proposition}\label{p:fsi-well-posed}
(a) Let $Y=\cH\times H^1(\Omega_s)\times L^2(\Omega_s)$ and $U:=L^2(\Gamma_s)$.
The initial-boundary value problem \eqref{e:dampedfsi-controlled} can be recast as the abstract Cauchy problem
\begin{equation} \label{e:cauchy-pbm}
\begin{cases} 
y'=\cA y+\cB g
\\[1mm]
y(0)=y_0\in Y\,,
\end{cases} 
\end{equation}
where the linear (dynamics and control, respectively) operators $\cA$ and $\cB$ 
are defined as follows: 

\begin{align}
&
\begin{cases}
& \cD(\cA)=\big\{(u,w,z)^T\in Y\colon \;
u\in V\,, \; \text{div}\,\sigma(w)\in L^2(\Omega_s)\,, \; z\in H^1(\Omega_s)\,, 
\\[1mm]
&\myspace
u-N\sigma(w)\cdot \nu\in \cD(A)\,, \; [z-\sigma(w)\cdot \nu]\big|_{\Gamma_s}=u\big|_{\Gamma_s}
\big\}
\\[3mm]
& \cA \begin{pmatrix} u\\[1mm]w\\[1mm]z\end{pmatrix}
=
\begin{pmatrix}
A\big(u-N\sigma(w)\cdot \nu\big)
\\[1mm]
z
\\[1mm]
\text{div}\,\sigma(w)-w
\end{pmatrix}\,,
\end{cases}
\label{e:operator-a}
\\[3mm] 
& \qquad \cB g =\begin{pmatrix} -AN\,g\\[1mm] 0 \\[1mm] 0\end{pmatrix}\,.
\label{e:operator-b}
\end{align}

\smallskip
\noindent
(b) The operator $\cA\colon \cD(\cA)\subset Y \longrightarrow Y$ is the infinitesimal
generator of a strongly continuous semigroup $\{e^{t\cA}\}_{t\ge 0}$ of contractions
in $Y$, while $\cB\in \cL(U,[\cD({\cA}^*)]')$.
(The equation in \eqref{e:cauchy-pbm} is initially understood in $[\cD({\cA}^*)]'$.)  
In addition, the semigroup $e^{t\cA}$ is exponentially stable, namely, there exist constants $C, \omega>0$ such that $\|e^{t\cA}\|_{\cL(Y)}\le C e^{-\omega t}$, $t\ge 0$.
\end{proposition}

\begin{proof}
\hspace{1mm} 
(a) 
The reformulation of the initial-boundary value problem \eqref{e:dampedfsi-controlled} 
as the Cauchy problem \eqref{e:cauchy-pbm}, with $\cA$ and $\cB$ defined by \eqref{e:operator-a} and \eqref{e:operator-b}, respectively, comes along with its variational formulation and the discussion
preceding the Proposition's statement.   
\\
(b) 
That $\cA$ is the infinitesimal generator of a strongly continuous semigroup of contractions
$\{e^{t\cA}\}_{t\ge 0}$ in $Y$ is proved following -- {\em mutatis mutandis} -- the proof of Proposition~3.1 in \cite{barbu-etal_2007}.
To confirm that $\cB\in \cL(U,[\cD({\cA}^*)]')$, it suffices to compute $\cA^{-1}\cB$.
It is readily verified that $\cA^{-1}\cB g=(-Ng,0,0)^T$ for $g\in U$, showing $\cA^{-1}\cB\in \cL(U,Y)$, which is equivalent to $\cB\in \cL(U,[\cD({\cA}^*)]')$.
Finally, uniform (exponential) stability of the free dynamics generator was shown in \cite[Theorem~1.2]{avalos-trig-dampedfsi_2009}.  
\end{proof}

\begin{remarks}
\begin{rm}
We recall that the stability analysis for the uncontrolled FSI \eqref{e:dampedfsi-controlled}
pursued in \cite{avalos-trig-dampedfsi_2009} is based on a different semigroup setup (devised 
by the same authors in \cite{avalos-trig-fsi_2009}), whose free dynamics generator
arises after finding that the pressure $p$ solves a suitable elliptic boundary problem and hence can be represented via proper Green maps.
(This idea has proved very effective also in the analysis of very different FSI, such as the one studied
in \cite{avalos-bucci-fsi1,avalos-bucci-fsi2}.)
As it will appear clearer in the next subsection, multipliers/energy methods allow to attain a dissipation relation -- {\em viz.}~(1.24) in \cite{avalos-trig-dampedfsi_2009} -- which readily yields the 
trace regularity estimate $\sigma(w)\cdot \nu\in L^2(0,T;L^2(\Gamma))$
for the hyperbolic component of the damped FSI; see \eqref{e:basic-reg-stress} below.
We note that this was not the case for the undamped FSI, i.e. \eqref{e:dampedfsi-free} with $a_2=a_1=0$.
Indeed, the respective regularity $\sigma(w)\cdot \nu\in L^2(0,T;H^{-1/2}(\Gamma))$ 
for the normal component of the stress tensor on the interface 
is not intrinsic to -- and not disclosed by -- the semigroup framework of 
\cite{avalos-trig-fsi_2009}.
It was {\em embedded in the very definition of weak solutions} to the FSI and shown to hold true
in \cite{barbu-etal_2007}, utilizing both nonlinear semigroup theory and microlocal analysis arguments. 
\end{rm}
\end{remarks}

\subsection{Boundary regularity of the normal stresses}
In this subsection we render explicit the regularity of the normal 
component of the stress tensor on the interface $\Gamma_s$, for strong and strict solutions of the {\em uncontrolled} FSI \eqref{e:dampedfsi-controlled}, as well as for solutions corresponding to initial data in intermediate spaces between the state space $Y$ and the domain of the (free dynamics) generator.
The distinct trace regularity results constitute fundamental elements for the subsequent analysis of the boundary regularity of the fluid variable $u$ and $u_t$.

\begin{proposition}\label{p:damped-traces}
Consider the FSI \eqref{e:dampedfsi-controlled} with $g\equiv 0$.
Then, the following statements are valid.
(i) For initial data $(u_0,w_0,w_1)=:y_0$ which belong to the energy space $Y$, the corresponding semigroup (weak) solutions $y(t):=(u(t),w(t),w_t(t))=e^{\cA t} y_0$ are such that
\begin{equation}\label{e:basic-reg-stress}
\sigma(w)\cdot \nu \in L^2(0,T;L^2(\Gamma_s)) \qquad \forall T>0\,.
\end{equation}

\noindent
(ii) For initial data $(u_0,w_0,z_0)$ which belong to the domain $\cD(\cA)$ of the dynamics generator, the corresponding solutions $(u(t),w(t),w_t(t))$ are such that
\begin{equation}\label{e:improved-reg-stress}
\sigma(w)\cdot \nu \in C([0,T],H^{1/2}(\Gamma_s))\cap H^1(0,T;L^2(\Gamma_s)) 
\qquad \forall T>0\,.
\end{equation}

\noindent
(iii) Let $0<\epsilon<1$. 
For initial data $(u_0,w_0,z_0)$ which belong to $\cD(\cA^\epsilon)$,
the corresponding semigroup solutions $(u(t),w(t),w_t(t))$ satisfy
\begin{equation}\label{e:interpolate}
\sigma(w)\cdot \nu \in L^2(0,T;H^{\epsilon/2}(\Gamma_s))
\cap H^\epsilon(0,T;L^2(\Gamma_s)) \qquad \forall T>0\,.
\end{equation}

\end{proposition}

\begin{proof}
Assume $y_0=(u_0,w_0,w_1)\in Y$. 
The first statement follows from a dissipation identity, {\em viz.} formula (1.24)
in \cite{avalos-trig-dampedfsi_2009}. 
The said equality -- which can be proved by using multipliers methods -- clarifies
in particular that for any $s,t$ with $0\le s\le t$ it holds
\begin{equation*}
\int_s^t |\sigma(w(\tau,\cdot))\cdot\nu|_{\Gamma_s}^2\,d\tau\le C\|y_0\|_Y^2\,;
\end{equation*}
taking $s=0$ and $t=T>0$ we obtain \eqref{e:basic-reg-stress}.

\smallskip
\noindent
For the second assertion, we assume $y_0\in \cD(\cA)$ and observe that
semigroup theory implies $y=(u,w,w_t)\in C([0,T],\cD(\cA))$.  
By the very definition of the domain of the free dynamics generator $\cA$ in
\eqref{e:operator-a}, we know that 
\begin{equation*}
u\in C([0,T],V)\subset C([0,T],H^1(\Omega_f))\,, \quad
w_t\in C([0,T],H^1(\Omega_s))\,.
\end{equation*}
Sobolev trace theory then obtains 
$u|_{\Gamma_s}\,, \, w_t|_{\Gamma_s}\in C([0,T],H^{1/2}(\Gamma_s))$.
Thus, using once more the definition of $\cD(\cA)$ and specifically
the interface condition, we see that
\begin{equation}\label{e:improved-reg-stress_1}
\sigma(w)\cdot \nu\big|_{\Gamma_s}=\big[w_t-u\big]\big|_{\Gamma_s}
= w_t|_{\Gamma_s}-u|_{\Gamma_s}\in C([0,T],H^{1/2}(\Gamma_s))\,.
\end{equation}
In addition, still with $y_0\in \cD(\cA)$, since the generator commutes with the
semigroup, $\sigma(w_t)\cdot \nu$ possesses the very same regularity than $\sigma(w)\cdot \nu$ with $y_0\in Y$; namely, $\sigma(w_t)\cdot \nu\in L^2(0,T;L^2(\Gamma_s))$.
The latter membership, combined with \eqref{e:improved-reg-stress_1},
confirms \eqref{e:improved-reg-stress}.

\smallskip
\noindent
The third statement, namely, \eqref{e:interpolate}, follows from \eqref{e:basic-reg-stress} and \eqref{e:improved-reg-stress} via interpolation.
\end{proof}

\smallskip
Thus, while moving towards the proof of our main result, we record
a few regularity results pertaining to the mild solution 
\begin{equation}\label{e:convolution} 
z(t)=\int_0^t e^{(t-s)A}f(s)\,ds\,
\end{equation}
to the Cauchy problem
\begin{equation*}
\begin{cases}
z'=Az+f
\\[1mm]
z(0)=0
\end{cases}
\end{equation*}
with $L^2$ (in time) affine term, in the case $\{e^{tA}\}_{t\ge 0}$ is an analytic semigroup on a Hilbert space $Y$.
One has
\begin{equation}\label{e:analytic-estimates}
f\in L^2(0,T;Y) \Longrightarrow
\begin{cases}
& z\in L^2(0,T;\cD(A))\,,
\\[1mm]
& z\in C([0,T],(\cD(A),Y)_{\frac{1}{2}})\,,
\\[1mm]
& z\in C([0,T],\cD((-A)^{1/2-\sigma}))\,, \qquad 0<\sigma<\frac{1}{2}\,. 
\end{cases}
\end{equation}
 
\begin{remark} \label{r:control-theoretic}
\begin{rm}
We note that the first statement is an instance of the parabolic regularity which dates back to de Simon \cite{desimon_64}; it constitutes by now a basic result of the theory of analytic operator semigroups.
A neat proof of this implication is given by Lasiecka in \cite[Appendix~A]{las_1980}.
Instead, the second statement is a consequence of an ``intermediate derivative theorem'' by 
Lions and Magenes \cite[Theorem~2.3, p.~15]{lions-mag_72}.
As a general reference on analytic semigroup and optimal regularity for parabolic problems, 
{\em cf.}~Lunardi's monograph \cite{lunardi-book}.
\end{rm}
\end{remark}

\subsection{Proof of Theorem~\ref{t:main} (Main result)}
{\bf \em i)} 
Assume $y_0=(u_0,w_0,w_1)\in Y$.
By Proposition~\ref{p:fsi-well-posed} we know that the fluid variable $u(t)$ 
is the mild solution of the equation $u'=Au-AN\sigma(w)\cdot \nu$; 
therefore, it is given by the formula 
\begin{equation*}
u(t)=e^{tA}u_0-A\int_0^t e^{(t-s)A} N\sigma(w(s))\cdot \nu\,ds\,.
\end{equation*}
Apply the Dirichlet trace operator to $u(t)$, recalling that it coincides with $-N^*A$ ({\em cf.}~\cite[Proposition~4.3]{las-tuff-fsi_2009}, also recorded as Lemma~A.1 in \cite{bucci-las-fsi_2010}),
to find
\begin{equation*}
\begin{split}
u(t)|_{\Gamma_s}&=-N^*A u(t) =
-N^*A e^{tA}u_0 + N^*A\int_0^t e^{(t-s)A} AN\sigma(w(s))\cdot \nu\,ds
\\[1mm]
&=:T_1(t)+T_2(t)\,.
\end{split}
\end{equation*}
For the first summand we readily obtain 
\begin{equation} \label{e:1a}
|T_1(t)|_{\Gamma_s}%=\|N^*Ae^{A t}u_0\|
\le \|N^*A^{3/4-\delta}\|\,\big\|A^{1/4+\delta}e^{A t}u_0\big\|
\le C\,\frac{e^{-\eta t}}{t^{1/4+\delta}}|u_0|\,,
\end{equation}
valid for all $t>0$ and suitable constants $C,\,\eta>0$.
(In the last inequality it has been used that the Stokes semigroup is analytic, as well as
that the solutions to the uncoupled Stokes flow decay exponentially.)

The analysis of the summand $T_2(t)$ simplifies % significantly 
over the one carried out in \cite[Proof of Theorem~2.10]{bucci-las-fsi_2010} in the undamped case ($a_2=0$).
This is because the improved regularity of the hyperbolic component of the system brought about by the presence of the feedback stabilizer 
-- and singled out in Proposition~\ref{p:damped-traces} -- overcomes the further splitting required by the application of the trickier Lemma~A.2 in \cite{bucci-las-fsi_2010}.
Thus, it will suffice to follow the computations performed on the term $u_{22}$ 
therein to achieve the sought regularity estimate; see~\cite[p.~228]{bucci-las-fsi_2010}.
Below we repeat the argument for the reader's convenience and the sake of completeness.

Rewrite $T_2(t)$ as follows:
\begin{equation*}
\begin{split}
T_2(t) &=N^*A\int_0^t e^{A (t-s)}AN\sigma(w(s))\cdot \nu\,ds
\\[1mm]
&=N^*A^{3/4-\delta}\int_0^t A^{1/2+2\delta}e^{A (t-s)}\,A^{3/4-\delta}N\sigma(w(s))\cdot \nu\,ds\,.
\end{split}
\end{equation*}
Because 
\begin{equation*}
A^{3/4-\delta}N\in \cL(L^2(\Gamma_s),L^2(\Omega_f))\,,
\quad N^*A^{3/4-\delta}\in \cL(L^2(\Omega_f),L^2(\Gamma_s))\,,
\end{equation*}
the boundary regularity (in time and space) of $T_2(t)$ on $\Gamma_s$ is
determined by the one possessed by the convolution term
\begin{equation*}
\int_0^t k(t-s)\,h(s)\,ds\,,
\end{equation*} 
specifically with
\begin{equation*}
k(s):=A^{1/2+2\delta}e^{A s}\,, \qquad 
h(s):= A^{3/4-\delta}N \sigma(w(s))\cdot \nu\,.
\end{equation*}
Recall now the basic asymptotic estimate (in a right neighbourhood of $s=0$)

\begin{equation*}
\|k(s)\|=\|A^{1/2+2\delta}e^{A s}\|\sim \frac{C}{s^{1/2+2\delta}}\,,
\end{equation*}
together with the basic regularity \eqref{e:basic-reg-stress} of $\sigma(w)\cdot\nu$, 
that is \eqref{e:basic-reg-stress} of Proposition~\ref{p:damped-traces}, to infer 

\begin{equation*}
k\in L^{2-\sigma}(0,T;L^2(\Omega_f))\,,\qquad
h\in L^2(0,T;L^2(\Omega_f))\,,
\end{equation*}
where $\sigma>0$ can be taken arbitrarily small.
By virtue of the Young inequality, it follows
\begin{equation*}
k\ast h\in L^r(0,T;L^2(\Omega_f))\,, 
\quad \frac{1}{r}=\frac{1}{2-\sigma}+\frac{1}{2}-1\,,
\end{equation*}
which establishes $k\ast h\in L^r(0,T;L^2(\Omega_f))$ for any finite $r\ge 1$.
In the present case then, we find
\begin{equation} \label{e:1b}
T_2\in L^r(0,T;L^2(\Gamma_s)) \qquad \forall r\ge 1\,.
\end{equation}
(Note that \eqref{e:1b} cannot be extended to $r=+\infty$, as $\sigma>0$.)
We conclude that the sought decomposition $u(t)=u_1(t)+u_2(t)$ is indeed simply the one with
\begin{equation*}
u_1(t):=e^{A t}u_0\,,
\qquad
u_2(t):=-A\int_0^t e^{A (t-s)}N\sigma(w(s))\cdot \nu\,ds\,,
\end{equation*}
with \eqref{e:1a} and \eqref{e:1b} confirming \eqref{e:se-damped} and \eqref{e:u_2-reg}, respectively.
 
\smallskip
\noindent
{\bf \em ii)} 
Assume now $y_0\in \cD(\cA^\epsilon)$ for some positive $\epsilon<1$.
From Proposition~\ref{p:damped-traces} we know that %the regularity result
\eqref{e:interpolate} holds.
Exploiting the improved space regularity
$\sigma(w)\cdot \nu\in L^2(0,T;H^{\epsilon/2}(\Gamma_s))$, we suitably rewrite
$u_2$ as follows:
\begin{equation*}
\begin{split}
u_2(t)&=-A\int_0^t e^{A (t-s)}N\sigma(w(s))\cdot \nu\,ds
\\[1mm]
&=-A^{1/4+\delta}A^{-\epsilon/4}
\int_0^t e^{A (t-s)}\,\underbrace{A^{\epsilon/4}\big[A^{3/4-\delta}N\big]\sigma(w(s))\cdot \nu}_{f(s)}\,ds\,.
\end{split}
\end{equation*}
Apply next the trace operator $-N^*A$ to find
\begin{equation}\label{e:last-rewrite}
u_2(t)\big|_{\Gamma_s}=[N^*A^{3/4-\delta}]A^{1/2+2\delta}A^{-\epsilon/4}
\int_0^t e^{A (t-s)}f(s)\,ds\,,
\end{equation}  
where $f\in L^2(0,T;L^2(\Omega_f))$.
On the basis of \eqref{e:last-rewrite}, we appeal to the latter estimate
in \eqref{e:analytic-estimates} to deduce that 
$u_2|_{\Gamma_s}\in C([0,T],L^2(\Gamma_s))$ i.e.~\eqref{e:u_2-suffices} holds true
($\delta$ must be suitably chosen: any $0<\delta<\epsilon/8$ will work).

The obtained trace regularity result \eqref{e:u_2-suffices} for the component $u_2$ can be actually
made more precise.
Indeed, return to \eqref{e:interpolate} which gives {\em a fortiori} 
$\sigma(w)\cdot \nu\in H^{\epsilon/2}(\Sigma_s)$. 
By way of the optimal parabolic regularity, 
$u_2\in H^{(\epsilon+3)/2,(\epsilon+3)/4}(Q_f)$ follows; 
thus, standard Sobolev trace theory yields 
\begin{equation*}
u_2\big|_{\Gamma_s}=-N^*Au_2\in H^{\epsilon/2+1,\epsilon/4+1/2}(\Sigma_s))
\subset H^{1/2+\epsilon/4}(0,T;L^2(\Gamma_s))\,, 
\end{equation*}
that is precisely the stronger \eqref{e:u_2-reg-improved}. 

\smallskip
\noindent
{\bf \em iii)} 
Take now $y_0\in \cD({\cA}^{1-\theta})$, $\theta\in (0,1)$.
On the one hand, the proof of the trace regularity result \eqref{e:u_t-reg} 
pertaining to the fluid acceleration $u_t$ can be streamlined, in comparison with
\cite[Proof of Theorem~2.10, (iii)]{bucci-las-fsi_2010}.
(This is again a consequence of the fact that \cite[Lemma~2.4]{bucci-las-fsi_2011}
is here replaced by Proposition~\ref{p:damped-traces}.)
On the other hand, we seek to disclose the improved ({\em space}) regularity 
\eqref{e:u_t-better-reg}, even though for the purpose of invoking the theory in
\cite{abl_2005,abl_2013} the membership $L^q(0,T;L^2(\Gamma_s))$ in \eqref{e:u_t-reg} would suffice. 
A shared tool for both proofs are interpolation result which also involve
the dual of Sobolev spaces with fractional exponents ({\em cf.}~\cite[Theorem~12.5]{lions-mag_72}).

We start from the expression 
\begin{equation*}
u_t(t)=\underbrace{Ae^{A t}u_0}_{V_1(t)}
+\underbrace{\frac{d}{dt}\int_0^t e^{A (t-s)}AN\sigma(w(s))\cdot \nu\,ds}_{V_2(t)}\,,
\end{equation*}
which is meant in the sense of distributions.
Assuming initially $y_0\in Y$, we compute 
\begin{equation*}
V_2(t)=AN \sigma(w(t))\cdot \nu
+A\int_0^t e^{A (t-s)}AN\sigma(w(s))\cdot \nu\,ds
=:V_{2a}(t)+V_{2b}(t)\,,
\end{equation*}
and pinpoint the regularity of either summand by rewriting and splitting as follows:
\begin{equation*}
\begin{split}
V_{2a}(t) &=A^{1/4+\delta_1}\big[A^{3/4-\delta_1}N\big] \sigma(w(t))\cdot \nu
\in L^2(0,T;[\cD(A^{1/4+\delta_1})]')\,,
\\[1mm]
V_{2b}(t) &=A\int_0^t e^{A(t-s)}AN\sigma(w(s))\cdot \nu\,ds
\\[1mm]
&=A^{1/4+\delta_1}\big(A\int_0^t e^{A(t-s)}\big[A^{3/4-\delta_1}N\big]\sigma(w(s))\cdot \nu\,ds\Big)
\in L^2(0,T;[\cD(A^{1/4+\delta_1})]')\,.
\end{split}
\end{equation*}
To infer the latter membership we simply used the first (basic) analytic estimate
in \eqref{e:analytic-estimates}.
Thus, the regularity of $V_{2a}$ and $V_{2b}$ combine to bring about
\begin{equation}\label{for-interpol_1}
y_0\in Y \Longrightarrow V_{2}\in L^2(0,T;[H^{1/2+2\delta_1}(\Omega_f)]')\,, 
\quad 0<\delta_1<\frac{1}{2}\,,
\end{equation}
where we used the identification $H^s(\Omega_f)\equiv \cD(A^{s/2})$ valid for $s<3/2$.
(We note that consistently with the improved regularity of the normal component
of the elastic tress tensor ({\em cf.}~Proposition~\ref{p:damped-traces}), 
\eqref{for-interpol_1} is a better regularity result than (3.18) in 
\cite{bucci-las-fsi_2010}.) 

\smallskip
\noindent
When $y_0\in \cD(\cA)$, the summand $V_2(t)$ is more conveniently rewritten in a different fashion:
\begin{equation*}
V_2(t)=V_{2a}(t)+V_{2b}(t)
=\underbrace{e^{A t}AN \sigma(w(0))\cdot \nu}_{V_{21}(t)}
+\underbrace{\int_0^t e^{A (t-s)}AN\sigma(w_t(s))\cdot \nu\,ds}_{V_{22}(t)}\,,
\end{equation*}
where by Proposition~\ref{p:damped-traces} and specifically in view of \eqref{e:improved-reg-stress} we know that
  
\begin{equation} \label{e:due-utili}
\sigma(w(0))\cdot \nu\in H^{1/2}(\Gamma_s)\,, \quad
\sigma(w_t)\cdot \nu\in L^2(0,T;L^2(\Gamma_s))\,.
\end{equation}
The regularity result on the right of \eqref{e:due-utili} suggests that we
rewrite the summand $V_{22}$ as follows:
\begin{equation*}
V_{22}(t)=A^{1/4+\delta_2}\int_0^t e^{(t-s)A}
\big[A^{3/4-\delta_2}N\big]\sigma(w_t(s))\cdot \nu\,ds
\end{equation*}
(valid for arbitrary $\delta_2\in (0,3/4)$).
Then, the first analytic estimate in \eqref{e:analytic-estimates} establishes
\begin{equation*}
A^{3/4-\delta_2}V_{22}\in L^2(0,T;L^2(\Omega_f))\,, 
\end{equation*}
that is
\begin{equation}\label{e:v_22}
V_{22}\in L^2(0,T;\cD(A^{3/4-\delta_2}))\equiv L^2(0,T;H^{3/2-2\delta_2}(\Omega_f))\,. 
\end{equation}

\smallskip
As for $V_{21}$, the regularity result on the left of \eqref{e:due-utili}
reveals that 
\begin{equation*}
V_{21}(t)=A^{\delta_2/2}\,e^{At}\chi\,, \quad
\chi:=A^{1/4}\big[A^{3/4-\delta_2/2}N\big]\sigma(w(0))\cdot \nu\in L^2(\Omega_f)\,.
\end{equation*}
which enables us to ascertain that
\begin{equation*}
A^{3/4-\delta_2}V_{21}=A^{3/4-\delta_2/2}e^{At}\chi\in L^{q_2}(0,T;L^2(\Omega_f)) 
\quad \text{provided $q_2(3/4-\delta_2/2)<1$;}
\end{equation*}
equivalently,
\begin{equation}\label{e:v_21}
V_{21}\in L^{q_2}(0,T;H^{3/2-2\delta_2}(\Omega_f))
\quad \text{for any $q_2<\frac{4}{3-2\delta_2}$.}
\end{equation}
We resume then \eqref{e:v_22} and observe that the upper bound $4/(3-2\delta_2)$
for the summability exponent $q_2$ is not greater than $2$ as long as
$\delta_2\le 1/2$; at the same time, $\delta_2$ must be taken arbitraily close to $0$ in order to attain the best space regularity.
Therefore, we limit $\delta_2\le 1/2$ and conclude from 
\eqref{e:v_21} and \eqref{e:v_22} (and still with $y_0\in \cD(\cA)$) that
\begin{equation}\label{e:for-interpol_2}
y_0\in \cD(\cA) \Longrightarrow 
V_2 %=V_{21}+V_{22}
\in L^{q_2}(0,T;H^{3/2-2\delta_2}(\Omega_f))
\quad \text{for any $q_2<\frac{4}{3-2\delta_2}$}
\end{equation}
holds true, for any given $\delta_2\le 1/2$.
 
A fundamental tool is now interpolation.
On the basis of \eqref{for-interpol_1} and \eqref{e:for-interpol_2} valid for 
$y_0\in Y$ and $y_0\in \cD(\cA)$, respectively, we utilize 
\cite[Theorem~12.5]{lions-mag_72} to establish
\begin{equation*}
y_0\in \cD({\cA}^{1-\theta}) \Longrightarrow 
V_2\in L^{q_2}(0,T;W) \qquad \text{for any $q_2<\frac{4}{3-2\delta_2}$,}
\end{equation*}
where $W$ is the Sobolev space
\begin{equation*}
W= \big(H^{3/2-2\delta_2}(\Omega_f),[H^{1/2+2\delta_1}(\Omega_f)]'\big)_{1-\theta,2}
=H^s(\Omega_f)\,,
\end{equation*}
with $s=(3/2-2\delta_2)(1-\theta)-\theta(1/2+2\delta_1)$,
which simplifies to 
$s=3/2-2(\theta+\delta)$ by setting $\delta=\delta_1=\delta_2$.
By standard trace theory, we conclude that 
\begin{equation} \label{e:interpol_1}
y_0\in \cD(\cA^{1-\theta}) \Longrightarrow 
V_2\big|_{\Gamma_s}\in L^{q_2}(0,T;H^{1-2\theta-2\delta}(\Gamma_s))\,, 
\quad \text{for any $q_2<\frac{4}{3-2\delta}$,}
\end{equation}
provided $s\ge 1/2$, which means 
\begin{equation}\label{e:constraint_1}
\delta+\theta\le \frac{1}{2}\,.
\end{equation}
We note that \eqref{e:constraint_1} forces $\theta<1/2$ (as well as $\delta <1/2$).

\smallskip
The analysis of $V_1$ is simpler. 
Since $y_0\in \cD(\cA^{1-\theta})$, for the first component one has 
\begin{equation*}
u_0\in (H^1(\Omega_f),L^2(\Omega_f))_\theta=H^{1-\theta}(\Omega_f)\equiv \cD(A^{(1-\theta)/2})\,,
\end{equation*}
where the identification is justified by the fact that $1-\theta<3/2$. 
We now proceed differently than in \cite[Proof of Theorem~2.10, (iii)]{bucci-las-fsi_2010}.
Aiming to pinpoint a better (space) regularity of $V_1$, we rewrite
$V_1(t)=Ae^{At}u_0=A^{(1+\theta)/2}e^{At}[A^{(1-\theta)/2}u_0]$ to find 
\begin{equation*}
A^{(1-\theta)/2-\sigma}V_1(t)=A^{1-\sigma}e^{At}[A^{(1-\theta)/2}u_0]\,,
\end{equation*}
which shows
\begin{equation}\label{e:before-trace}
y_0\in \cD(\cA^{1-\theta}) \Longrightarrow
V_1\in L^{q_1}(0,T;\cD(A^{(1-\theta)/2-2\sigma}))
=L^{q_1}(0,T;H^{1-\theta-2\sigma}(\Omega_f)) 
\end{equation}
for any $q_1<\frac{1}{1-\sigma}$.
Only subsequently we invoke trace theory: if $1-\theta-2\sigma\ge 1/2$, i.e.
\begin{equation}\label{e:constraint_2}
2\sigma+\theta\le \frac{1}{2}\,,
\end{equation}
the interior regularity \eqref{e:before-trace} implies
\begin{equation}\label{e:interpol_2}
y_0\in \cD(\cA^{1-\theta}) \Longrightarrow
V_1\big|_{\Gamma_s}\in L^{q_1}(0,T;H^{1/2-\theta-2\sigma}(\Gamma_s)) 
\quad \text{for any $q_1<\frac{1}{1-\sigma}$.}
\end{equation}
Since the role of $\delta$ and $\sigma$ is similar, and consistently with the constraints \eqref{e:constraint_1} and \eqref{e:constraint_2}, we set $\sigma=\delta/2$ in \eqref{e:interpol_2},
which becomes 
\begin{equation}\label{e:interpol_2bis}
y_0\in \cD(\cA^{1-\theta}) \Longrightarrow
V_1\big|_{\Gamma_s}\in L^{q_1}(0,T;H^{1/2-\theta-\delta}(\Gamma_s)) 
\quad \text{for any $q_1<\frac{2}{2-\delta}$.}
\end{equation}

A comparison between the Sobolev exponents in \eqref{e:interpol_2bis} and \eqref{e:interpol_1}
shows $$\frac{1}{2}-\theta-\delta\le 1-2\theta-2\delta$$ owing to \eqref{e:constraint_1};
in addition, readily $\frac{2}{2-\delta}<\frac{4}{3-2\delta}$. 
Therefore, we conclude that
\begin{equation*}
y_0\in \cD(\cA^{1-\theta}) \Longrightarrow V\big|_{\Gamma_s} \in L^q(0,T;H^{1/2-\theta-\delta}(\Gamma_s)) 
\quad \text{for any $q< \frac{2}{2-\delta}$,}
\end{equation*}
which is nothing but \eqref{e:u_t-better-reg}.
Finally, by taking $\delta=1/2-\theta$ we see that \eqref{e:u_t-better-reg}
yields \eqref{e:u_t-reg}, which ends the proof.
\qed

% SECTION 3

\section{Application to optimal boundary control} \label{s:application}
Regularity estimates are a critical part of most results in control theory for PDE.
In this section we show how the trace regularity results provided by Theorem~\ref{t:main} combine to bring about the well-posedness of both differential and algebraic Riccati equations (DRE and ARE,
respectively) corresponding to the quadratic optimal control problems on a finite and infinite time horizon
associated with the PDE problem \eqref{e:dampedfsi-controlled} 
(including, in particular, problem \eqref{e:dampedfsi-controlled}-\eqref{e:energy-cost} with $0<T\le +\infty$), thus confirming their {\em full} solvability.
We will ascertain that the Assumptions~1.1 and 1.4 in \cite{abl_2013} -- pertaining to the (dynamics, control, observation) operators involved -- are fulfilled.
An account of this technical issue, which will be kept brief, is given in Subsection~\ref{ss:conclusion}, thereby concluding the article.
Rather, we find it worth giving insight into the prerequisite role of certain trace regularity results within
the LQ-problem for hyperbolic PDE subject to boundary control, as well as for composite systems of hyperbolic-parabolic PDE (that is the case of the present work), along with some historical comments.
This is the content of the next subsection.

\subsection{The broader context}
Let $Y$ and $U$ be two separable Hilbert spaces, the {\em state} and {\em control} space, respectively.
Consider a linear control system such as the one in \eqref{e:cauchy-pbm}, under the following basic assumptions.

\begin{itemize}
\item
The closed linear operator $\cA\colon \cD(\cA)\subset Y \to Y$ is the infinitesimal generator of a strongly continuous semigroup $e^{t\cA}$ on $Y$, $t\ge 0$;
\item 
$\cB\in \cL(U,[\cD({\cA}^*)]')$.
\end{itemize}
Thus, given $y_0\in Y$, the Cauchy problem \eqref{e:cauchy-pbm} possesses a unique {\em mild} solution
given by 
\begin{equation} \label{e:mild}
y(t)= e^{t\cA}y_0+\int_0^t e^{(t-s)\cA}\cB g(s)\,ds\,,
\end{equation}
initially understood in the extrapolation space $[\cD({\cA}^*)]'$; 
see \cite[\S0.~3, p.~6, and Remark~7.1.2, p.~646]{las-trig-redbooks}.
To fully understand the regularity of the state \eqref{e:mild} one needs to pinpoint the one of the operator
\begin{equation*}
L\colon g(\cdot)\longrightarrow (Lg)(t):=\int_0^t e^{(t-s)\cA}\cB g(s)\,ds\,,
\end{equation*}
that is the (so called ``input-to-state'') mapping which associates to any control function
$g\in L^2(0,T;U)$ the solution to the Cauchy problem \eqref{e:cauchy-pbm} with $y_0=0$.
We note that the regularity properties of $L$ and its adjoint $L^*$ are, in turn, fully determined by the regularity (in time and space) of the kernel $e^{t\cA}\cB$ and its adjoint.
Pinpointing it may be a challenging task in the case of interest, where $\cB\notin \cL(U,Y)$. 
It is here that the study of the parabolic and the hyperbolic cases diverge, and where a glimpse
of the possibility for systems of coupled parabolic-hyperbolic PDE emerges. 

The property of being {\em analytic} possessed by the semigroup $e^{t\cA}$, combined with the 
structure of the operator $\cB$ (stemming from the modeling of boundary control forces)
results in $Lg\in L^2(0,T;Y)$ via the first one of \eqref{e:analytic-estimates},
thereby obtaining $y\in L^2(0,T;Y)$ for {\em any} state -- while $y\in C([0,T],Y)$ might be false, 
depending on the specific BC in place.
Furthermore, the very same property gives rise to well-posed Riccati equations with bounded gains.
This scenario pertains to parabolic (and parabolic-like) PDE; see \cite{las-trig-lncis} and 
\cite{bensoussan-etal}.
Instead, the fact that 
\begin{enumerate}
\item[i)]
the semigroup $e^{t\cA}$ that drives the free dynamics is {\em not} analytic, while
\item[ii)]
the control operator $\cB$ is still {\em unbounded} (it is necessarily so because of the control action exerted from the boundary),
\end{enumerate}
renders the hyperbolic case trickier.

To the state equation we associate the quadratic functional
\begin{equation} \label{e:cost}
J(u)=\int_0^T \left(\|\cR y(t)\|_Z^2 + \|g(t)\|_U^2\right)dt\,, 
\end{equation}
where $Z$ is a third separable Hilbert space -- the so called observation space (possibly,
$Z\equiv Y$).
A priori, the {\em observation} operator $\cR$ simply satisfies 
\begin{equation}\label{e:basic-for-r}
\cR\in \cL(Y,Z)\,.
\end{equation}
The formulation of the optimal control problem under study is the usual one.
\begin{problem}[\bf The optimal control problem] \label{p:lq-pbm}
Given $y_0\in Y$, seek a control function $\hat{g}$ which minimizes the
cost functional \eqref{e:cost} overall $g\in L^2(0,T;U)$, where $y(\cdot)=y(\cdot\,;y_0,g)$
is the solution to \eqref{e:cauchy-pbm} corresponding to $g(\cdot)$.
\end{problem}
We note however that by ``solving'' Problem~\ref{p:lq-pbm} it is meant that certain
principal facts should hold, besides the existence of a unique optimal pair 
$(\hat{g}(\cdot,s;y_0),\hat{y}(\cdot,s;y_0))$,
a property which simply follows via classical variational arguments. 
Namely, 
\begin{itemize}
\item[--]
that the optimal control $\hat{g}(t)$ admits a (pointwise in time) {\em feedback} representation,
in terms of the optimal state $\hat{y}(t)$;
\item[--]
that the optimal cost operator $P$ ($P(t)$, when $T<+\infty$) solves the corresponding 
algebraic (differential, respectively) Riccati equation, that is
\begin{equation} \label{e:are}
\begin{split}
& (Px,\cA z)_Y+(\cA x,Pz)_Y-(\cB^*Px,\cB^*Pz)_U+(Rx,Rz)_Z=0  
\\[1mm]
& \myspace\myspace \textrm{for any $x,z\in \cD(\cA)$;}
\end{split}
\end{equation}
thus, the issue of well-posedness of the ARE (DRE) arises, requiring  
\item[--]
that a meaning is given to the gain operator ${\cB}^*P$ (${\cB}^*P(t)$) either on the state space $Y$
(possibly by means of extensions), or -- which will be the case, here -- as a {\em bounded} operator
on a suitable dense subset of $Y$.
\end{itemize}
It is by now well known that the latter property might be lacking, in view of the aforesaid (intrinsic) 
features of the control system under examination. 
Now, the key role played by the kernel $e^{t\cA}\cB$ % (and by its adjoint ${\cB}^*e^{t{\cA}^*}$)
(or, possibly -- which is weaker --, by $Re^{t\cA}\cB$)
becomes even more evident upon recalling that the optimal cost operator $P$ is explicitly defined in terms of the optimal state $\hat{y}(t)=\Phi(t)y_0$, $0\le t<+\infty$, via the formula
\begin{equation*}
Py_0 := \int_0^\infty e^{t{\cA}^*}R^*R \Phi(t)y_0\, dt \qquad x\in Y\,.
\end{equation*}

It dates back to the beginning of the eighties the discovery that in the case of linear hyperbolic equations -- with the wave equation as a paradigm -- the following ({\em admissibility}) estimate
\begin{equation}\label{e:admissibility}
\exists C_T>0 \colon \quad \int_0^T \|{\cB}^*e^{t{\cA}^*}y_0\|_U^2\, dt\le C_T\|y_0\|_Y^2 \qquad \forall y_0\in Y\,,
\end{equation}
is the one to be ascertained.
Indeed, it implies $L\in \cL(L^2(0,T;U),C([0,T],Y))$, so that any state belongs to $C([0,T],Y)$.
In addition, it enables to show a well-posedness result for the corresponding DRE, provided $\cR$ is smoothing in an appropriate sense.
The PDE interpretation of the abstract condition \eqref{e:admissibility} is intriguing:
if the Cauchy problem \eqref{e:cauchy-pbm} represents the Cauchy-Dirichlet problem
\begin{equation}\label{e:ibvp-wave}%\tag{IBVP-wave}
\begin{cases}
w_{tt} =\Delta w
& \text{in $Q:=(0,T)\times \Omega$}
\\[1mm]
w(0,\cdot)=w_0\,,\; w_t(0,\cdot)=w_1
& \text{in $\Omega$}
\\[1mm]
w(t,x) =g & \text{on $\Sigma:=(0,T)\times \Gamma$} 
\end{cases}
\end{equation}
for a single, linear wave equation 
(where $y(t)=(w(t,\cdot),w_t(t,\cdot))$, $Y=L^2(\Omega) \times H^{-1}(\Omega)$ and $U=L^2(\Gamma)$), then the abstract condition
\eqref{e:admissibility} is equivalent to the trace regularity result
\begin{equation}
\exists C_T>0 \colon \quad \int_0^T \int_{\Gamma}\Big|\frac{\partial w}{\partial \nu}(t,x)\Big|^2\, d\sigma\,dt
\le C_T \int_{\Omega}\big(|\nabla w_0|^2+|w_1|^2\big)\,dx\,,
\end{equation}
valid for the solution to the IBVP \eqref{e:ibvp-wave} with initial data 
$(w_0,w_1)\in H^1_0(\Omega)\times L^2(\Omega)$ and (trivial) boundary datum $g\equiv 0$
({\em cf.} \cite{las-lions-trig_1986}).
The proof of the boundary regularity results more often rely on energy methods, along with an appropriate choice of multipliers which are suited to the specific PDE problem, or they may require pseudodifferential methods.
For a concise introduction (yet pretty comprehensive one, till the eighites) to the LQ-problem for either parabolic (and parabolic-like) or hyperbolic equations with boundary or point control, the reader is 
referred to \cite{las-trig-lncis}. 
For a thorough treatise and bibliography, see \cite{bensoussan-etal} and \cite{las-trig-redbooks}.

% SUBSUBSECTION

\subsubsection{Composite systems of PDE}
By the end of the nineties a theory of the LQ-problem % and of Riccati equations 
ensuring a bounded gain operator and well-posed Riccati equations -- and hence akin to the one pertaining to parabolic-like PDE -- has been devised, for abstract control 
systems which yield the (so called) {\em singular estimate} 
\begin{equation} \label{e:se}
\|e^{t\cA}\cB\|_{\cL(U,Y)}\sim \frac{C}{t^{\gamma}}
\end{equation}
in a right neighbourhood of $t=0$ (for some $\gamma\in (0,1)$, and $C>0$),
even in the absence of analyticity of the semigroup $e^{t\cA}$;
see \cite{las-trento}, and \cite{las-tuff-se-revisited_2009} for the Bolza problem.
A breakthrough in this direction came from the study \cite{avalos-las_1996} of an optimal control problem for an acoustic-structure interaction PDE system.
 
The project carried out and accomplished in the studies \cite{abl_2005,abl_2013} (with a few additions and refinements that are forthcoming) provides a framework for optimal control problems with quadratic functionals for a wider class of PDE systems which comprise both hyperbolic and parabolic components, with the latter subjected to boundary/interface control.
This class has proven to be sufficiently general to encompass 
widely different PDE models such as certain
thermoelastic systems, acoustic-structure and fluid-structure interactions
({\em cf.}~\cite{abl-thermo_2005}, \cite{bucci-applicationes_2008}, 
\cite{bucci-las-fsi_2010,bucci-las-fsi_2011}).
The achievements of \cite{abl_2005,abl_2013} include
\begin{itemize}
\item
the feedback synthesis of the optimal control, along with

\item
well-posed (differential, first, and algebraic, next) Riccati equations, 
\end{itemize} 
with a {\em gain} operator which is bounded on a dense subset of the state space. 
This theory allows non-smoothing observations $\cR$.
The aforesaid conclusions, whose relevance is well-understood within systems theory,
raise highly technical issues in the context of composite PDE systems,
both on a functional-analytic level (see in particular \cite[Section~5]{abl_2013})
and on the `PDE level', when it comes to the verification of the control-theoretic
properties that characterize the couple $(A,B)$ (these are Assumptions~1.4 in \cite{abl_2013}). 
When dealing with systems of coupled parabolic-hyperbolic PDE one needs to disclose
and exploit the regularity of the hyperbolic traces in order to eventually show the sought
regularity of the parabolic ones -- 
a more challenging task especially when the coupling occurrs either on a sub-boundary of the domain or
at an interface, like in the case of the FSI under consideration here.

% SUBSECTION

\subsection{Optimal boundary control of the present FSI} 
\label{ss:conclusion}
For the linearization of the original FSI studied by Barbu~{\em et al.} (i.e. for \eqref{e:dampedfsi-free} with $a_2=a_1=0$), subject to a control force acting on the interface, the study of Lasiecka and Tuffaha \cite{las-tuff-fsi_2009} revealed that the sought estimate \eqref{e:se} does not hold, unless the observation
operator is suitably smoothing.
On the other hand, the weaker estimate
\begin{equation*}
\|\cO e^{t\cA}\cB\|_{\cL(U,Z)}\sim \frac{C}{t^{\gamma}}
\end{equation*}
is established (with suitable $\cO$ and $\gamma\in (0,1)$), which is nothing but a singular estimate in a weaker topology than the one of the energy space, and which introduces a limitation on the cost functionals allowed. 
But then the ``revisited singular estimate'' theory in \cite{las-tuff-se-revisited_2009} applies,
with the plus of the Bolza problem being feasible. 
(We used above the usual notation $e^{t\cA}$ and $\cB$ for the semigroup that describes the overall free dynamics and the control operator.)

As explained in the Introduction, with the focus on the optimal control problem on a finite time horizon
for the very same linear FSI, the analysis pursued in \cite{bucci-las-fsi_2010,bucci-las-fsi_2011} succeded in removing the constraints on the cost functionals, by showing that the FSI falls in the more general class and set-up devised in \cite{abl_2005}.
The present work continues and completes the PDE analysis of \cite{bucci-las-fsi_2010},
by exploring the infinite (besides finite) time horizon case.
Given the stability properties of the uncontrolled FSI, we were necessarily led to study a variant of
the original linearization, that is \eqref{e:dampedfsi-free} with $a_2>0$ (and $a_1>0$, for the sake of simplicity).

\begin{remark} \label{r:pde-theoretic}
\begin{rm}
All the boundary regularity results contained in Theorem~\ref{t:main} can be interpreted as respective 
regularity properties of the operators ${\cB}^*e^{t{\cA}^*}$ or ${\cB}^*e^{t{\cA}^*}{{\cA}^*}^\epsilon$,
showing that the requirements of \cite[Assumptions~1.1 and 1.4]{abl_2013} are met.
% \cR^*\cR 
This is accomplished making use of pretty much the same arguments that were used in the end of \cite[Section~4]{bucci-las-fsi_2010} for the undamped FSI.
Once again, the observation operator $\cR$ is allowed to be the identity, just like in the undamped case.
Thus in particular, the cost functional \eqref{e:energy-cost} 
(that is integral of the full quadratic energy of the system)
fits into the picture.
The major steps of the needed verification are briefly illustrated below, thereby concluding the article.
\end{rm}
\end{remark}

\noindent
{\em Validity of the Assumptions~1.1 and 1.4 in \cite{abl_2013}, as a consequence of Theorem~\ref{t:main}.}
Let us start from the basic Assumptions~1.1.
Just observe that these are found among the statements of Proposition~\ref{p:fsi-well-posed},
with the property of exponential stability of the semigroup $e^{t\cA}$ proved in \cite{avalos-trig-dampedfsi_2009}.
To check the Assumptions~1.4, the {\em incipit} is as follows: in view of \eqref{e:operator-b} and
\cite[Proposition~4.3]{las-tuff-fsi_2009} 
\begin{equation*}
\cB^*e^{t\cA^*}y_0=-N^*A\hat{u}(t)=\hat{u}(t)\big|_{\Gamma_s}\,,
\end{equation*}
having denoted by $\hat{u}(t)$ the first component of the solution 
$\hat{y}(t)=(\hat{u}(t),\hat{w}(t)\hat{w}_t(t)$ to the {\em uncontrolled} Cauchy problem
\begin{equation*}
\begin{cases}
\hat{y}'=\cA^* \hat{y}
\\[1mm]
\hat{y}(0)=y_0\,.
\end{cases}
\end{equation*}
Thus, the sought decomposition of the operator $\cB^*e^{t\cA^*}$, along with the regularity estimates
in Assumption~1.4(i)-(ii) in \cite{abl_2013} are confirmed by \eqref{e:se-damped} (with $\gamma=1/4+\delta$) and \eqref{e:u_2-reg} of Theorem~\ref{t:main}, respectively. 

Next, \eqref{e:u_2-suffices} establishes Assumption~1.4(iiia), initially with
arbitrary $\epsilon\in (0,1)$.
By following the argument and computations in the very end of \cite[Section~4]{bucci-las-fsi_2010},
and given the obtained regularity \eqref{e:u_t-reg} for the fluid acceleration, we see that Assumption~1.4(iiic) is satisfied, with $\epsilon=\theta<1/2$, provided 
\begin{equation} \label{e:maintains} 
\cR^*\cR \in \cL(\cD(\cA^\epsilon),\cD({\cA^*}^\epsilon))\,.
\end{equation}
We note that the latter is not a requirement that the observation operator should possess an appropriate smoothing property; in particular, \eqref{e:maintains} holds true when $\cR\equiv I$.
As already explained in \cite[Remark~2.8]{bucci-las-fsi_2010}, this is because 
$\cD(\cA^\epsilon)\equiv \cD({\cA^*}^\epsilon)$ for $\epsilon>0$ sufficiently small;
and in fact, in view of Theorem~\ref{t:main}, $\epsilon$ can be taken {\em arbitrarily} small.
(Recall that the equivalence between $\cD(\cA^\epsilon)$ and $\cD({\cA^*}^\epsilon)$ for $\epsilon>0$ sufficiently small is a consequence of $\cD(\cA)=\cD({\cA^*})$, as the domains of fractional powers of the generator $\cA$ ($\cA^*$) are intermediate spaces between $\cD(\cA)$ ($\cD({\cA^*}^\epsilon)$) and $Y$.) 

In view of the above, all the statements of \cite[Theorem~1.5]{abl_2013} -- pertaining to the infinite time horizon case -- are valid. 
In addition, since the said Assumptions~1.1 and 1.4 in \cite{abl_2013} contain the ones needed to solve the optimal control problem on a finite time interval, that are Hypotesis~2.1 and 2.2. in \cite{abl_2005},
{\em a fortiori} \cite[Theorem~2.3]{abl_2005} applies as well.
Both differential and algebraic Riccati equations are solvable, and the synthesis of the (unique)
optimal control is guaranteed.
 
% REFERENCES

% The END

\end{document}